\documentclass[11pt,reqno]{amsart}

\usepackage{amssymb,amsthm,amscd,array}
\usepackage[all]{xy}

\let\ssection=\section
\renewcommand{\section}{\setcounter{equation}{0}\ssection}

\newcommand {\emptycomment}[1]{}

\newcommand{\bbZ}{\mathbb{Z}}

\newcommand{\Hom}{\mathrm{Hom}}
\newcommand{\End}{\mathrm{End}}

\newcommand{\half}{\textstyle{\frac{1}{2}}}
\newcommand{\fg}{\mathfrak{g}}

\newcommand{\fa}{\mathfrak{a}}
\newcommand{\fb}{\mathfrak{b}}

\newcommand{\fn}{\mathfrak{n}}

\DeclareMathOperator{\Ker}{Ker}

\newtheorem{thm}{Theorem}
\newtheorem{lemma}{Lemma}[section]

\newtheorem{prop}[lemma]{Proposition}

\newtheorem{exe}[lemma]{Example}
\newtheorem{defi}[lemma]{Definition}

\def\a{\alpha}
\def\b{\beta}

\def\e{\varepsilon}

\def\r{\rho}

\allowdisplaybreaks

\begin{document}

\title{Crossed modules for Hom-Lie antialgebras}

\author[T. Zhang]{Tao Zhang}\thanks{Corresponding author(Tao Zhang):~~zhangtao@htu.edu.cn}
\address{College of Mathematics and Information Science\\
Henan Normal University\\
Xinxiang 453007, PR China}
\email{zhangtao@htu.edu.cn}

\author[H. Zhang]{Heyu Zhang}
\address{College of Mathematics and Information Science\\
Henan Normal University\\
Xinxiang 453007, PR China}
\email{zhy199404@126.com}

\date{}

\begin{abstract}
  In this paper, we introduced the concept of crossed modules for Hom-Lie antialgebras.
  It is proved that the category of crossed modules for Hom-Lie antialgebras and the category of  $Cat^1$-Hom-Lie antialgebras are equivalent to each other.  The relationship between the crossed modules extension of Hom-Lie antialgebras and the third cohomology group are investigated.
\end{abstract}

\subjclass[2010]{Primary 17D99; Secondary 18G60}

\keywords{Hom-Lie antialgebra; action; semidirect product; crossed module}

\maketitle

\thispagestyle{empty}


\section{Introduction}
The notion of Lie antialgebras was introduced by V. Ovsienko in \cite{Ovs} as an algebraic structure in the context of symplectic and contact geometry of $\bbZ_2$-graded space.  A Lie antialgebra is a $\bbZ_2$-graded vector space $\fa=\fa_0\oplus\fa_1$ where $\fa_0$ is a commutative associative algebra acting on $\fa_1$ as a derivation satisfying some compatible conditions, see Definition \ref{def001}.
The universal enveloping algebra and representation theory of Lie antialgebras have been investigated \cite{LM}.

Crossed modules of groups were introduced by Whitehead in the late 1940s as algebraic models for path-connected CW-spaces, see \cite{Wh}.
Crossed modules of Lie algebras and associative algebras have also been investigated by various authors, see \cite{Bau,Ell,Ger,KL,Lue,Wag}.
It is well known that equivalent classes of crossed module extensions of associative algebras and Lie algebras are in one-to-one correspondence with elements of their Chevally-Eilenberg and Hochschild third cohomology groups.
For crossed modules of algebroids, Leibniz algebras, Hom-Lie algebras, Hom-Lie-Rinehart algebras and internal crossed modules in  semi-abelian categories, see \cite{Alp,B1,B2,Cas,CM,Jan1,Jan2,Zhang}.

In a recent paper \cite{ZZ}, we introduced the concept of Hom-Lie antialgebras which is a Hom-analogue of Ovsienko's Lie antialgebras.
It worth mentioning that the concept of Hom-Lie algebras were invented by J. Hartwig, D. Larsson and S. Silvestrov \cite{HLS} in order to better understand $q$-deformations of the classical conformal algebras (such as the Virasoro algebra). Similarly,  Hom-Lie antialgebras can be seen as some deformations of conformal Lie antialgebras.
A thorough study of Hom-Lie antialgebra structures can also explain some properties of underlying Hom-Lie superalgebras.
The present investigation can bring some light in this direction.

Also in \cite{ZZ},  the general representations and cohomology theory of Hom-Lie antialgebras are investigated.
We proved that the equivalent classes of abelian extensions of Hom-Lie antialgebras are classified by the second cohomology group.
It is then natural to ask: What about the third cohomology group of Hom-Lie antialgebras?
Is there some crossed module extension structures of Hom-Lie antialgebras that can be related to third cohomology group?
In this paper, we give a partial answer to these questions.

The key point is to give a suitable definition of crossed module for Hom-Lie antialgebras.
To solve this problem, first we give a detailed study on the actions and semidirect products of Hom-Lie antialgebras.
Then we introduce the notion of crossed module for Hom-Lie antialgebras and $Cat^1$-Hom-Lie antialgebras.
It is proved that the category of crossed modules for Hom-Lie antialgebras and the category of $Cat^1$-Hom-Lie antialgebras are equivalent to each other.
Finally, we introduced the notion of crossed module extensions of Hom-Lie antialgebras.
Given an equivalent class of crossed module extensions, we prove that there is a canonical element in the third cohomology group.

The paper is organized as follows. In Section 2, we revisit some definitions and notations of Hom-Lie antialgebras. In Section 3, we study the representation and semidirect product of Hom-Lie antialgebras. In Section 4, we introduced the concept of crossed module for Hom-Lie antialgebras and $Cat^1$-Hom-Lie antialgebras.  We prove that the category of crossed modules for Hom-Lie antialgebras and the category of $Cat^1$-Hom-Lie antialgebras are equivalent to each other. In Section 5, we define crossed module extension of Hom-Lie antialgebras.
The relationship between the crossed module extension of Hom-Lie antialgebras and the third cohomology group are found.

Throughout this paper, we work with an algebraically closed field  of characteristic 0.
For a  $\bbZ_2$-graded vector space (also called super vector space) $V=V_0\oplus V_1$ , we consider the standard  $\bbZ_2$-grading on the algebra of linear maps on  $V$: $\End(V)=\End(V)_0\oplus \End(V)_1$ where $\End(V)_0=\Hom(V_0,V_0)\oplus\Hom(V_1,V_1)$ and $\End(V)_1=\Hom(V_0,V_1)\oplus\Hom(V_1,V_0)$.
\section{Hom-Lie antialgebras}\label{defintion}
\begin{defi}\label{def001}\cite{Ovs}
A Lie antialgebra is a supercommutative $\bbZ_2$-graded algebra
$\fa=\fa_0\oplus\fa_1$ such that the following identities hold:
\begin{eqnarray}
\label{AssCommT0}
x_1\cdot\left(x_2\cdot x_3\right)&=&\left(x_1\cdot x_2\right)\cdot x_3,\\
\label{CacT0}
x_1\cdot(x_2\cdot y_1)&=&\half(x_1\cdot x_2)\cdot y_1,\\
\label{ICommT0}
x_1\cdot[y_1,y_2]&=&[x_1\cdot y_1,\,  y_2]\;+\;[y_1,\, x_1\cdot y_2],\\
\label{Jack0}
y_1\cdot [y_2,y_3]&+&y_2 \cdot[y_3,y_1]\;+\;y_3\cdot [y_1,y_2]=0,
\end{eqnarray}
for all homogeneous elements $x_1,x_2,x_3\in\fa_0,~ y_1, y_2, y_3\in\fa_1$.
Note that, different from \cite{Ovs}, we denote by square brackets the product of homogeneous elements of degree 1, since it is skew symmetric.
\end{defi}

\begin{defi}\cite{Ovs}
Let $\fa$ and $\widetilde{\fa}$ be Lie antialgebras. A Lie antialgebra homomorphism $\phi$ from $\fa$ to $\widetilde{\fa}$ consists of two linear maps  $\phi_0:\fa_0\rightarrow\widetilde{\fa_0}$, $\phi_1:\fa_1\rightarrow\widetilde{\fa_1}$, such that the following conditions hold :
\begin{eqnarray*}
\label{Thomo1}
\phi_0(x_1\cdot x_2)&=&\phi_0(x_1)\cdot\phi_0(x_2),\\
\label{Thomo2}
\phi_1(x_1\cdot y_1)&=&\phi_0(x_1)\cdot\phi_1(y_1),\\
\label{Thomo3}
\phi_0([y_1,y_2])&=&[\phi_1(y_1),\phi_1(y_2)],
\end{eqnarray*}
for all $x_1,x_2,x_3\in\fa_0,~ y_1, y_2, y_3\in\fa_1$.
\end{defi}

\begin{defi}\label{def1}\cite{ZZ}
(1) A Hom-Lie antialgebra $(\fa,\a,\b)$ is a supercommutative $\bbZ_2$-graded algebra:
$\fa=\fa_0\oplus\fa_1$, together with two linear maps $(\alpha,\beta)$: $\alpha:\fa_0\rightarrow \fa_0,~\beta:\fa_1\rightarrow \fa_1$,
satisfying the following identities:
\begin{eqnarray}
\label{HLie-anti01}
\a(x_1)\cdot\left(x_2\cdot x_3\right)&=&\left(x_1\cdot x_2\right)\cdot\a(x_3),\\
\label{HLie-anti02}
\a(x_1)\cdot(x_2\cdot y_1)&=&\half(x_1\cdot x_2)\cdot\b(y_1),\\
\label{HLie-anti03}
\a(x_1)\cdot[y_1,y_2]&=&[(x_1\cdot y_1),\b(y_2)]\;+\;[\b(y_1),(x_1\cdot y_2)],\\
\label{HLie-anti04}
\b(y_1)\cdot[y_2,y_3]&+&\b(y_2)\cdot[y_3,y_1]\;+\;\b(y_3)\cdot[y_1,y_2]=0,
\end{eqnarray}
for all $x_1,x_2,x_3\in\fa_0,~ y_1, y_2, y_3\in\fa_1$.

(2) A Hom-Lie antialgebra is called multiplicative if $(\a,\b)$ form an algebraic homomorphism of $\fa$, i.e. for any $x_1, x_2\in\fa_0,~ y_1, y_2\in\fa_1$, we have
\begin{eqnarray*}
\a(x_1\cdot x_2)&=&\a(x_1)\cdot\a(x_2),\\
\b(x_1\cdot y_1)&=&\a(x_1)\cdot\b(y_1),\\
\a([y_1,y_2])&=&[\b(y_1),\b(y_2)].
\end{eqnarray*}

(3) A Hom-Lie subantialgebra $(\fa',\a',\b')$ of $(\fa,\a,\b)$ is a subspace $\fa'\subseteq \fa$ such that $(\fa',\a',\b')$  is itself a Hom-Lie antialgebra under operations of $\fa$ restrict to $\fa'$ and $\a'=\a|_{\fa'},\b'=\b|_{\fa'}$.
\end{defi}
The Hom-Lie antialgebras in this paper are assumed to be multiplicative.

\begin{exe}\cite{ZZ}
{\rm Consider the 3-dimensional Hom-Lie antialgebra $K_3$ as follows.
This algebra has the basis $\{\e; a,b\}$, where $\e$ is even and $a,b$ are odd,
Consider the linear map $(\a,\b):\fa\to \fa$ defined by
$$\a(\e)=\e, \quad \b(a)=\mu a,\quad \b(b)=\mu^{-1} b$$
on the basis elements, then we obtain a Hom-Lie antialgebra structure given by
\begin{eqnarray*}
\label{aslA}
\e\cdot{}\e=\e,\quad
\e\cdot{}a=\half\,\mu\, a,\quad
\e\cdot{}b=\textstyle{\frac{1}{2}}\mu^{-1} b,\quad
{[a,b]}=\half\,\e.
\end{eqnarray*}
}
\end{exe}

\begin{exe}\cite{ZZ}
\label{ExMain}
{\rm
Another example of a Hom-Lie antialgebra is the
\textit{conformal Hom-Lie antialgebra}  $K(1)$.
This is a simple infinite-dimensional Hom-Lie antialgebra with the basis
$$
\textstyle
\left\{
\e_n,\;n\in\bbZ;
~a_i, ~i\in\bbZ+\half
\right\},
$$
where $\e_n$ are even, $a_i$ are odd, and
$\a(\e_i)=\e_i$, $\b(a_i)=(1+q^i) a_i$
satisfy the following relations:
\begin{eqnarray*}
\label{GhosRel}
\e_n\cdot{}\e_m&=&\e_{n+m},\\
\e_n\cdot{}a_i &=&\half (1+q^i) a_{n+i},\\
{[a_i, a_j]}&=&\half\left(\{j\}-\{i\}\right)\e_{i+j},
\end{eqnarray*}
where $\{i\}=(q^i-1)/(q-1),\  q\neq 1$.
}
\end{exe}

\begin{defi}
Let $(\fa,\a,\b)$ and $(\widetilde{\fa},\widetilde{\a},\widetilde{\b})$ be Hom-Lie antialgebras. A Hom-Lie antialgebra homomorphism $\phi=(\phi_0,\phi_1)$ from $\fa$ to $\widetilde{\fa}$ consists of two linear maps  $\phi_0:\fa_0\rightarrow\widetilde{\fa_0}$, $\phi_1:\fa_1\rightarrow\widetilde{\fa_1}$, such that the following equalities hold for all $x_i\in\fa_0,~y_i\in\fa_1$:
\begin{eqnarray*}
\phi_0\circ\a&=&\widetilde{\a}\circ\phi_0,\\
\label{Thomo5}
\phi_1\circ\b&=&\widetilde{\b}\circ\phi_1,\\
\label{Thomo1}
\phi_0(x_1\cdot x_2)&=&\phi_0(x_1)\cdot\phi_0(x_2),\\
\label{Thomo2}
\phi_1(x_1\cdot y_1)&=&\phi_0(x_1)\cdot\phi_1(y_1),\\
\label{Thomo3}
\phi_0([y_1,y_2])&=&[\phi_1(y_1),\phi_1(y_2)].
\end{eqnarray*}
\end{defi}

\begin{defi}
Let $(\fa,\alpha,\beta)$ be a Hom-Lie antialgebra. A subalgebra $\fb$ of $\fa$ is a subspace of $\fa$, which is closed for the structural operations and invariant by $\alpha,\beta$, that is,
\begin{enumerate}
\item [(i)] $\fb_0 \cdot \fb_0\subseteq\fb_0, \fb_0 \cdot \fb_1\subseteq\fb_1, [\fb_1,\fb_1]\subseteq \fb_0$,
\item [(ii)] $\alpha(\fb_0)\subseteq \fb_0, \beta(\fb_1)\subseteq \fb_1$.
\end{enumerate}
We call a subalgebra $\fb$ of $\fa$ to be an ideal  if
\begin{enumerate}
\item [(iii)] $\fb_0 \cdot \fa_0\subseteq\fb_0,\, \fb_0 \cdot \fa_1\subseteq\fb_1,\, \fb_1 \cdot \fa_0\subseteq\fb_1$ and $[\fb_1, \fa_1]\subseteq \fb_0$.
\end{enumerate}
\end{defi}
\section{Semidirect product of Hom-Lie antialgebras}\label{Rep}
In this section, we introduce the concept of representation and action of Hom-Lie antialgebras.
Then we construct the semidirect product of Hom-Lie antialgebras.
\begin{defi}
Let $(\fa,\a,\b)$ be a Hom-Lie antialgebra, $V=V_0\oplus V_1$  be a Hom-super vector space (a super vector space with linear maps $\a_{V_0}\in\Hom(V_0,V_0),~\b_{V_1}\in\Hom(V_1,V_1)$). A representation of $(\fa,\a,\b)$ over the Hom-super vector space $V$ is a pair of linear maps $\r=(\r_0,\r_1):\r_0:\fa_0\rightarrow \End(V)_0,~\r_1:\fa_1\rightarrow \End(V)_1$ such that the following conditions hold:
\begin{eqnarray}
\label{01}
\a_{V_0}(\r_0(x_1)(u_1))&=&\r_0(\a(x_1))\a_{V_0}(u_1),\\
\label{02}
\b_{V_1}(\r_0(x_1)(w_1))&=&\r_0(\a(x_1))\b_{V_1}(w_1),\\
\label{03}
\a_{V_0}(\r_1(y_1)(u_1))&=&\r_0(\b(y_1))\a_{V_0}(u_1),\\
\label{04}
\b_{V_1}(\r_1(y_1)(w_1))&=&\r_0(\b(y_1))\b_{V_1}(w_1),\\
\label{a}
\r_0(\a(x_1))\circ\r_0(x_2)(u_1)&=&\r_0(x_1\cdot x_2)\circ\a_{V_0}(u_1),\\
\label{b}
\r_0(\a(x_1))\circ\r_0(x_2)(w_1)&=&\half\r_0(x_1\cdot x_2)\circ\b_{V_1}(w_1),\\
\label{c}
\r_0(\a(x_1))\circ\r_1(y_1)(u_1)&=&\half\r_1(\b(y_1))\circ\r_0(x_1)(u_1),\\
\label{d}
\r_1(x_1\cdot y_1)\circ\a_{V_0}(u_1)&=&\half\r_1(\b(y_1))\circ\r_0(x_1)(u_1),\\
\label{e}
\r_0(\a(x_1))\circ\r_1(y_1)(w_1)&=&\r_1(x_1\cdot y_1)\circ\b_{V_1}(w_1)+\r_1(\b(y_1))\circ\r_0(x_1)(w_1),\\
\label{f}
\r_0([y_1,y_2])\circ\a_{V_0}(u_1)&=&\r_1(\b(y_1))\circ\r_1(y_2)(u_1)-\r_1(\b(y_2))\circ\r_1(y_1)(u_1),\\
\label{g}
\r_0([y_1,y_2])\circ\b_{V_1}(w_1)&=&-\r_1(\b(y_1))\circ\r_1(y_2)(w_1)+\r_1(\b(y_2))\circ\r_1(y_1)(w_1),
\end{eqnarray}
for all $x_1,x_2\in\fa_0,~y_1,y_2\in\fa_1$, $u_1\in V_0,~w_1\in V_1$.
\end{defi}

It is easy to see that for any Hom-Lie antialgebra $(\fa,\a,\b)$,  there is a natural adjoint representation on itself and its ideal.
\begin{prop}\label{prop-rep}
Let $(\fa,\a,\b)$ be a Hom-Lie antialgebra, $(V,\a_{V_0},\b_{V_1})$ be a Hom-super vector space. Then
$\r=(\r_0,\r_1)$  is a representation of $(\fa,\a,\b)$ over $(V,\a_{V_0},\b_{V_1})$  if and only if $\fa\oplus V\triangleq(\fa_0\oplus V_0)\oplus(\fa_1\oplus V_1)$ is a Hom-Lie antialgebra under the following operations:
\begin{eqnarray}\label{homo10}
(\a+\a_{V_0})(x_1,u_1)&=&(\a(x_1),\a_{V_0}(u_1)),\\
\label{homo20}
(\b+\b_{V_1})(y_1,w_1)&=&(\b(y_1),\b_{V_1}(w_1)),\\
\label{op10}
(x_1,u_1)\cdot(x_2,u_2)&=&(x_1\cdot x_2,~ \r_0(x_1)(u_2)+\r_0(x_2)(u_1)),\\
\label{op20}
(x_1,u_1)\cdot(y_1,w_1)&=&(x_1\cdot y_1,~ \r_0(x_1)(w_1)+\r_1(y_1)(u_1)),\\
\label{op30}
[(y_1,w_1),(y_2,w_2)]&=&([y_1,y_2], \r_1(y_1)(w_2)-\r_1(y_2)(w_1)),
\end{eqnarray}
for all $x_1,x_2\in\fa_0,~y_1,y_2\in\fa_1$, $u_1,u_2\in V_0,~w_1,w_2\in V_1$.
\end{prop}
The proof of the above Proposition \ref{prop-rep} is by easy direct computations, so we omit the details.

From now on, we assume $(V,\alpha_{V_0},\beta_{V_1})$ is a representation of $(\fa,\alpha,\beta)$ unless otherwise stated.
If moreover $(V,\alpha_{V_0},\beta_{V_1})$ is itself a Hom-Lie antialgebra under some compatible conditions  with the representation conditions, we can construct a general Hom-Lie antialgebra $\fa\oplus V$ on the direct sum space as follows.

\begin{defi}
Let $(V,\alpha_{V_0},\beta_{V_1})$ and $(\fa,\alpha,\beta)$ be two Hom-Lie antialgebras.
An action of $(\fa,\alpha,\beta)$ over $(V,\alpha_{V_0},\beta_{V_1})$  is a representation
$(V,\rho)$ such that the following conditions hold:
\begin{eqnarray}
\label{action01}
\rho_0(\alpha(x_1))(u_1\cdot u_2)&=&\rho_0(x_1)(u_1)\cdot\alpha(u_2),\\
\label{action021}
\rho_0(\alpha(x_1))(u_1\cdot w_1)&=&\half\rho_0(x_1)(u_1)\cdot\beta_{V_1}(w_1),\\
\label{action022}
\rho_1(y_1)(u_2)\cdot\alpha_{V_0}(u_1)&=&\half\rho_1(\beta(y_1))(u_1\cdot u_2),\\
\label{action023}
\rho_0(x_1)(w_1)\cdot\alpha_{V_0}(u_1)&=&\half\rho_0(x_1)(u_1)\cdot\beta_{V_1}(w_1),\\
\label{action031}
\rho_0(\alpha(x_1))([w_1,w_2])&=&[\rho_0(x_1)(w_1),\beta_{V_1}(w_2)]+[\beta_{V_1}(w_1),\rho_0(x_1)(w_2)],\\
\label{action032}
\rho_1(\beta(y_1))(u_1\cdot w_1)&=&\alpha_{V_0}(u_1)\cdot\rho_1(y_1)(w_1)-[\rho_1(y_1)(u_1),\beta_{V_1}(w_1)],\\
\label{action04}
\rho_1(\beta(y_1))([w_1,w_2])&=&\beta_{V_1}(w_1)\cdot\rho_1(y_1)( w_2)-\beta_{V_1}(w_2)\cdot \rho_1(y_1)(w_1),
\end{eqnarray}
for all $x_1\in\fa_0,~y_1\in\fa_1$, $u_1,u_2\in V_0,~w_1,w_2\in V_1$.
\end{defi}

Now we can obtain our main construction of Hom-Lie antialgebra structure.
\begin{thm}\label{Thm1}
Let $(V,\alpha_{V_0},\beta_{V_1})$ and $(\fa,\alpha,\beta)$ be two Hom-Lie antialgebras with an action of $(\fa,\alpha,\beta)$ over $(V,\alpha_{V_0},\beta_{V_1})$.
 Then $(\fa\oplus V,\a+\a_{V_0},\b+\b_{V_1})$ is a Hom-Lie antialgebra under the following maps:
\begin{eqnarray}
\label{homo1}(\a+\a_{V_0})(x_1,u_1)&=&(\a(x_1),\a_{V_0}(u_1)),\notag\\
\label{homo2}(\b+\b_{V_1})(y_1,w_1)&=&(\b(y_1),\b_{V_1}(w_1)),\notag\\
\label{op1}
(x_1,u_1)\cdot(x_2,u_2)&=&(x_1\cdot x_2,~ \r_0(x_1)(u_2)+\r_0(x_2)(u_1)
+u_1\cdot u_2),\notag\\
\label{op2}
(x_1,u_1)\cdot(y_1,w_1)&=&(x_1\cdot y_1,~ \r_0(x_1)(w_1)+\r_1(y_1)(u_1)
+u_1\cdot w_1),\notag\\
\label{op3}
[(y_1,w_1),(y_2,w_2)]&=&([y_1,y_2],~ \r_1(y_1)(w_2)-\r_1(y_2)(w_1)
+[w_1,w_2]),\notag
\end{eqnarray}
for all $x_1,x_2\in\fa_0,~y_1,y_2\in\fa_1, u_1,u_2\in V_0, w_1,w_2\in V_1$.
This is called a semidirect product of $\fa$ and $V$,
denoted by $\fa\ltimes V$.
\end{thm}

\begin{proof}
By definition, we have
\begin{eqnarray*}
&&(\a+\a_{V_0})((x_1,u_1))\cdot((x_2,u_2)\cdot(x_3,u_3))\\
&=&(\a(x_1),\a_{V_0}(u_1))\cdot((x_2,u_2)\cdot(x_3,u_3))\\
&=&(\a(x_1)\cdot(x_2\cdot x_3),
~\r_0(\a(x_1))(\r_0(x_2)(u_3))
+\r_0(\a(x_1))(\r_0(x_3)(u_2))\\
&&+\r_0(x_2\cdot x_3)(\a_{V_0}(u_1))
+\underbrace{\r_0(\a(x_1))(u_2\cdot u_3)}_A\\
&&+\underbrace{\a_{V_0}(u_1)\cdot\r_0(x_2)(u_3)}_B+\underbrace{\a_{V_0}(u_1)\cdot\r_0(x_3)(u_2)}_C+\a_{V_0}(u_1)\cdot(u_2\cdot u_3)
\end{eqnarray*}
and
\begin{eqnarray*}
&&((x_1,u_1)\cdot(x_2,u_2))\cdot(\a+\a_{V_0})((x_3,u_3))\\
&=&((x_1,u_1)\cdot(x_2,u_2))\cdot(\a(x_3),\a_{V_0}(u_3))\\
&=&((x_1\cdot x_2)\cdot \a(x_3),
~\r_0(x_1\cdot x_2)(\a_{V_0}(u_3))
+\r_0(\a(x_3))\r_0(x_1)(u_2)\\
&&+\r_0(\a(x_3))\r_0(x_2)(u_1)
+\underbrace{\r_0(\a(x_3))(u_1\cdot u_2)}_{C'}.\\
&&+\underbrace{\r_0(x_1)(u_2)\cdot\a_{V_0}(u_3)}_{A'}+\underbrace{\r_0(x_2)(u_1)\cdot\a_{V_0}(u_3)}_{B'}+(u_1\cdot u_2)\cdot\a_{V_0}(u_3).
\end{eqnarray*}
Due to \eqref{action01} and commutativity of product in $V_0$, we have
$$A=A', ~~B=B', ~~C=C'.$$
Thus we obtain
\begin{eqnarray}\label{semidirect01}\notag
&&(\a+\a_{V_0})((x_1,u_1))\cdot((x_2,u_2)\cdot(x_3,u_3))\\
&=&((x_1,u_1)\cdot(x_2,u_2))\cdot(\a+\a_{V_0})((x_3,u_3)).
\end{eqnarray}

For \eqref{HLie-anti02}, we have
\begin{eqnarray*}
&&(\a+\a_{V_0})((x_1,u_1))\cdot((x_2,u_2)\cdot(y_1,w_1))\\
&=&(\a(x_1),\a_{V_0}(u_1))\cdot((x_2,u_2)\cdot(y_1,w_1))\\
&=&(\a(x_1)\cdot(x_2,y_1),
~\r_0(\a(x_1))\r_0(x_2)(w_1)
+\r_0(\a(x_1))\r_1(y_1)(u_2)\\
&&+\r_1(x_2\cdot y_1)(\a_{V_0}(u_1))+\underbrace{\r_0(\a(x_1))(u_2\cdot w_1)}_D\\
&&+\underbrace{\a_{V_0}(u_1)\cdot\r_0(x_2)(w_1)}_{E}+\underbrace{\a_{V_0}(u_1)\cdot\r_1(y_1)(u_2)}_{F}+\a_{V_0}(u_1)\cdot(u_2\cdot w_1)
\end{eqnarray*}
\begin{eqnarray*}
&&(x_1,u_1)\cdot(x_2,u_2))\cdot(\b+\b_{V_1})(y_1,w_1)\\
&=&((x_1,u_1)\cdot(x_2,u_2))\cdot(\b(y_1),\b_{V_1}(w_1))\\
&=&((x_1,x_2)\cdot\b(y_1),
~\r_0(x_1\cdot x_2)(\b_{V_1}(w_1))
+\r_1(\b(y_1))\r_0(x_1)(u_2)\\
&&+\r_1(\b(y_1))\r_0(x_2)(u_1)
+\underbrace{\r_1(\b(y_1))(u_1\cdot u_2)}_{F'}\\
&&+\underbrace{\r_0(x_1)(u_2)\cdot\b_{V_1}(w_1)}_{D'}+\underbrace{\r_0(x_2)(u_1)\cdot\b_{V_1}(w_1)}_{E'}+(u_1\cdot u_2)\cdot\b_{V_1}(w_1))
\end{eqnarray*}
Due to \eqref{action021} and \eqref{action022}, we have
$$D=\half D',~ E=\half E', ~F=\half F'.$$
Thus we obtain
\begin{eqnarray}\label{semidirect02}\notag
&&(\a+\a_{V_0})((x_1,u_1))\cdot((x_2,u_2)\cdot(y_1,w_1))\\
&=&\half((x_1,u_1)\cdot(x_2,u_2))\cdot(\b+\b_{V_1})(y_1,w_1).
\end{eqnarray}

For \eqref{HLie-anti03}, we have
\begin{eqnarray*}
&&(\a+\a_{V_0})(x_1,u_1)\cdot[(y_1,w_1),(y_2,w_2)]\\
&=&(\a(x_1),\a_{V_0}(u_1))\cdot[(y_1,w_1),(y_2,w_2)]\\
&=&(\a(x_1)\cdot[y_1,y_2],
~\r_0(\a(x_1))\r_1(y_1)(w_2)
-\r_0(\a(x_1))\r_1(y_2)(w_1)\\
&&+\r_0([y_1,y_2])(\a_{V_0}(u_1))
+\underbrace{\r_0(\a(x_1))([w_1,w_2])}_{G}\\
&&+\underbrace{\a_{V_0}(u_1)\cdot\r_1(y_1)(w_2)}_{H}-\underbrace{\a_{V_0}(u_1)\cdot\r_1(y_2)(w_1)}_{I}+\a_{V_0}(u_1)\cdot[w_1,w_2])
\end{eqnarray*}
\begin{eqnarray*}
&&[(x_1,u_1)\cdot(y_1,w_1),(\b+\b_{V_1})(y_2,w_2)]\\
&=&[(x_1,u_1)\cdot(y_1,w_1)),(\b(y_2),\b_{V_1}(w_2))]\\
&=&\big([(x_1\cdot y_1),\b(y_2)],
~\r_1(x_1\cdot y_1)\cdot \b_{V_1}(w_2)
-\r_1(\b(y_2))\r_0(x_1)(w_1)\\
&&-\r_1(\b(y_2))\r_1(y_1)(u_1)
\underbrace{-\r_1(\b(y_2))(u_1\cdot w_1)}_{I_1}\\
&&+\underbrace{[\r_0(x_1)(w_1),\b_{V_1}(w_2)]}_{G_1}+\underbrace{[\r_1(y_1)(u_1),\b_{V_1}(w_2)]}_{H_1}+[u_1\cdot w_1,\b_{V_1}(w_2)]\big)
\end{eqnarray*}
\begin{eqnarray*}
&&[(\b+\b_{V_1})(y_1,w_1),(x_1,u_1)\cdot(y_2,w_2)]\\
&=&[(\b(y_1),\b_{V_1}(w_1)),(x_1,u_1)\cdot(y_2,w_2)]\\
&=&\big([\b(y_1),(x_1\cdot y_2)],
\r_1(\b(y_1))\r_0(x_1)(w_2)
-\r_1(x_1\cdot y_2)\b_{V_1}(w_1)\\
&&+\r_1(\b(y_1))\r_1(y_2)(u_1)
+\underbrace{\r_1(\b(y_1))(u_1\cdot w_2)}_{H_2}\\
&&+\underbrace{[\b_{V_1}(w_1),\r_0(x_1)(w_2)]}_{G_2}+\underbrace{[\b_{V_1}(w_1),\r_1(y_2)(u_1)]}_{I_2}+[\b_{V_1}(w_1),u_1\cdot w_2]\big)
\end{eqnarray*}
Due to \eqref{action031} and \eqref{action032}, we have
$$G=G_1+G_2,~ H=H_1+H_2, ~ I=I_1+I_2.$$
Thus we obtain
\begin{eqnarray}\label{semidirect03}\notag
&&(\a+\a_{V_0})(x_1,u_1)\cdot[(y_1,w_1),(y_2,w_2)]\\
\notag
&=&[(x_1,u_1)\cdot(y_1,w_1),(\b+\b_{V_1})(y_2,w_2)]\\
&&+[(\b+\b_{V_1})(y_1,w_1),(x_1,u_1)\cdot(y_2,w_2)].
\end{eqnarray}

For \eqref{HLie-anti04}, by definition \eqref{op3} we have
\begin{eqnarray*}
&&(\b+\b_{V_1})(y_1,w_1)\cdot[(y_2,w_2),(y_3,w_3)]\\
&=&(\b(y_1),\b_{V_1}(w_1))\cdot[(y_2,w_2),(y_3,w_3)]\\
&=&\big(\b(y_1)\cdot[y_2,y_3],
~\r_1(\b(y_1))(\r_1(y_2)(w_3))
-\r_1(\b(y_1))\r_1(y_3)(w_2)\\
&&+\r_0([y_2,y_3])(\b_{V_1}(w_1))+
\underbrace{\r_1(\b(y_1))([w_2,w_3])}_{J}\\
&&+\underbrace{\r_1(y_2)(w_3)\cdot\b_{V_1}(w_1)}_{K}-
\underbrace{\r_1(y_3)(w_2)\cdot\b_{V_1}(w_1)}_{L}+\b_{V_1}(w_1)\cdot[w_2,w_3]\big)
\end{eqnarray*}
\begin{eqnarray*}
&&(\b+\b_{V_1})(y_2,w_2)\cdot[(y_3,w_3),(y_1,w_1)]\\
&=&(\b(y_2),\b_{V_1}(w_2))\cdot[(y_3,w_3),(y_1,w_1)\\
&=&(\b(y_2)\cdot[y_3,y_1],
~\r_1(\b(y_2))(\r_1(y_3)(w_1))
-\r_1(\b(y_2))\r_1(y_1)(w_3)\\
&&+\r_0([y_3,y_1])(\b_{V_1}(w_2))+\underbrace{\r_1(\b(y_2))([w_3,w_1])}_{K_1}\\
&&+\underbrace{\r_1(y_3)(w_1)\cdot\b_{w_1}(w_2)}_{L_1}-\underbrace{\r_1(y_1)(w_3)\cdot\b_{V_1}(w_2)}_{J_1}+\b_{V_1}(w_2)\cdot[w_3,w_1]\big)
\end{eqnarray*}
\begin{eqnarray*}
&&(\b+\b_{V_1})(y_3,w_3)\cdot[(y_1,w_1),(y_2,w_2)]\\
&=&\big(\b(y_3),\b_{V_1}(w_3))\cdot[(y_1,w_1),(y_2,w_2)]\\
&=&\big(\b(y_3)\cdot[y_1,y_2],
~\r_1(\b(y_3))(\r_1(y_1)(w_2))
-\r_1(\b(y_3))\r_1(y_2)(w_1)\\
&&+\r_0([y_1,y_2])(\b_{V_1}(w_3))+\underbrace{\r_1(\b(y_3))([w_1,w_2])}_{L_2}\\
&&\underbrace{\r_1(y_1)(w_2))\cdot\b_{V_1}(w_3)}_{J_2}-\underbrace{\r_1(y_2)(w_1)\cdot\b_{V_1}(w_3)}_{K_2}+\b_{V_1}(w_3)\cdot[w_1,w_2]\big)
\end{eqnarray*}
Due to \eqref{action031} and \eqref{action032}, we have
$$J+J_1+J_2=0,~ K+K_1+K_2=0, ~L+L_1+L_2=0.$$
Thus we obtain
\begin{eqnarray}\label{semidirect04}\notag
&&(\b+\b_{V_1})(y_1,w_1)\cdot[(y_2,w_2),(y_3,w_3)]\\
\notag
&&+(\b+\b_{V_1})(y_2,w_2)\cdot[(y_3,w_3),(y_1,w_1)]\\
&&+(\b+\b_{V_1})(y_3,w_3)\cdot[(y_1,w_1),(y_2,w_2)]=0.
\end{eqnarray}
From \eqref{semidirect01}--\eqref{semidirect04}, we obtain that $\fa\ltimes V$ is a Hom-Lie antialgebra.
The proof is completed.
\end{proof}

\section{Crossed modules for Hom-Lie antialgebras}
In this section, we introduce the concept of crossed module for Hom-Lie antialgebras which can be constructed from semidirect product.
Then we establish the relationships between the category of crossed modules for Hom-Lie antialgebras and the category of $Cat^1$-Hom-Lie antialgebras.

\begin{defi}\label{CR} Let $(V,\alpha_{V_0},\beta_{V_1})$ and $(\fa,\alpha,\beta)$ be two Hom-Lie antialgebras with an action of  $(\fa,\alpha,\beta)$ over $(V,\alpha_{V_0},\beta_{V_1})$.
A crossed module of Hom-Lie antialgebras $\partial:(V,\alpha_{V_0},\beta_{V_1})\rightarrow(\fa,\alpha,\beta)$ is an Hom-Lie antialgebra homomorphism such that the following identities hold:
\begin{eqnarray}
\label{cm1}
\partial_0\circ\rho_0(x)(u_1)&=&x\cdot\partial_0(u_1),\\
\label{cm2}
\partial_1\circ\rho_0(x)(w_1)&=&x\cdot\partial_1(w_1),\\
\label{cm3}
\partial_1\circ\rho_1(y)(u_1)&=&y\cdot\partial_0(u_1),\\
\label{cm4}
\partial_0\circ\rho_1(y)(w_1)&=&[y,\partial_1(w_1)],\\
\label{pei1}
\rho_0(\partial_0(u_1))(u_2)&=&u_1\cdot u_2,\\
\label{pei2}
\rho_1(\partial_1(w_1))(w_2)&=&[w_1,w_2],\\
\label{pei3}
\rho_0(\partial_0(u_1))(w_1)&=&\rho_1(\partial_1(w_1))(u_1)=u_1\cdot w_1,
\end{eqnarray}
for all $x\in\fa_0$, $u_1,u_2\in V_0$, $y\in\fa_1$, $w_1,w_2\in V_1$.
\end{defi}

We will denote a crossed module by  $(V,\fa, \partial)$  or $\partial:V\rightarrow\fa$ in the following text for simplicity.

\begin{exe}
If $\fb$ is an ideal of a Hom-Lie antialgebra $\fa$, then $\fa$ acts on $\fb$ by the adjoint
representation. $(\fb,\fa, i)$ is a crossed module, where $i:\fb\rightarrow\fa$  is the inclusion map.
In particular, for any Hom-Lie antialgebra $\fa$, the triples $(\fa,\fa, id)$ and $(0,\fa, 0)$ are crossed modules.
\end{exe}

\begin{exe}\cite{ZZ2}
Let $(\fb,\a_{\fb_0},\b_{\fb_1})$, $(\fa,\a,\b)$ and $(\widetilde{\fa},\widetilde{\a},\widetilde{\b})$ be Hom-Lie antialgebras. An extension of $\fa$ by $\fb$ is a short exact sequence
\begin{equation*}
\xymatrix@C=0.5cm{
  0 \ar[r]^{} & \fb  \ar[rr]^{i } && \widetilde{\fa } \ar[rr]^{\pi} && \fa  \ar[r]^{} & 0 }
\end{equation*}
of Hom-Lie antialgebras. It is  called a central extension, if $\fb$ is contained in the center $Z(\widetilde{\fa})$ of $\widetilde{\fa}$.
Each central extension gives a crossed module $(\widetilde{\fa }, \fa,\pi)$ where $\pi$ is the projective map.
\end{exe}

\begin{defi}\label{COM}
Let $\partial:(V,\alpha_{V_0},\beta_{V_1})\rightarrow(\fa,\alpha,\beta)$ and $\widetilde{\partial}:(\widetilde{V},\widetilde{\alpha}_{V_0},\widetilde{\beta}_{V_1})\rightarrow(\widetilde{\fa},\widetilde{\alpha},\widetilde{\beta})$ be two crossed modules of Hom-Lie antialgebras. A morphism of crossed modules is a pair of Hom-Lie antialgebra homomorphisms $\phi: V\to \widetilde{V}$ and $\psi:\fa\to \widetilde{\fa}$ such that the following conditions hold:
\begin{eqnarray*}
\phi_0\circ\partial_0=\widetilde{\partial}_0\circ\phi_0,\quad\psi_1\circ\partial_1=\widetilde{\partial}_1\circ\psi_1.
\end{eqnarray*}
\end{defi}
We will denote by $\mathbf{XHLA}$ the category of crossed modules of Hom-Lie antialgebras.

\begin{defi}
A $Cat^1$-Hom-Lie antialgebra $\left(M,N,s,t\right)$ consists of a Hom-Lie antialgebra $(M,\alpha_{M_0},\beta_{M_1})$ together with a Hom-Lie subantialgebra $(N,\alpha_{N_0},\beta_{N_1})$ and the structural homomorphisms: $s=(s_0, s_1),~~t=(t_0,t_1):(M,\alpha_{M_0},\beta_{M_1})\rightarrow (N,\alpha_{N_0},\beta_{N_1})$ such that
\begin{eqnarray}
\label{cat01}&&s_0|_{N_0}=t_0|_{N_0}=id_{N_0},\\
\label{cat02}&&s_1|_{N_1}=t_1|_{N_1}=id_{N_1},\\
\label{cat1}&&\Ker s_0\cdot \Ker t_0=0,\\
\label{cat2}&&\Ker s_0\cdot \Ker t_1=0,\\
\label{cat3}&&\Ker t_0\cdot \Ker s_1=0,\\
\label{cat4}&&[\Ker s_1,\Ker t_1]=0,
\end{eqnarray}
where $id_{N_0}:{N_0}\to {N_0}$ and $id_{N_1}:{N_1}\to {N_1}$ are the identity maps,
$\Ker s$ and $\Ker t$ are the kernel spaces of $s$ and $t$.
\end{defi}

\begin{defi}
Let $\left(M,N,s,t\right)$ and $\left(\widetilde{M},\widetilde{N},\widetilde{s},\widetilde{t}\right)$ be two $Cat^1$-Hom-Lie antialgebras.
A morphism of $Cat^1$-Hom-Lie antiagebras $f:\left(M,N,s,t\right)\to\left(\widetilde{M},\widetilde{N},\widetilde{s},\widetilde{t}\right)$ is a Hom-Lie antialgebra homomorphism  $(f_0,f_1):(M,\alpha_{M_0},\beta_{M_1})\to (\widetilde{M},\widetilde{\alpha}_{M_0},\widetilde{\beta}_{M_1})$ such that $f_0(N_0)\subseteq \widetilde{N_0}$, $f_1(N_1)\subseteq \widetilde{N_1}$, and
 \begin{eqnarray*}
 \widetilde{s_0}\circ f_0&=&f_0|_{N_0}\circ s_0,\quad
 \widetilde{t_0}\circ f_0=f_0|_{N_0}\circ t_0, \\
 \widetilde{s_1}\circ f_1&=&f_1|_{N_1}\circ s_1,\quad
 \widetilde{t_1}\circ f_1=f_1|_{N_1}\circ t_1.
 \end{eqnarray*}
\end{defi}
We will denote by $\mathbf{CHLA}$ the category of $Cat^1$-Hom-Lie antialgebras.

\begin{exe}
Any Hom-Lie antialgebra $(\fa,\alpha,\beta)$ can be  considered as a $Cat^1$-Hom-Lie antialgebra with itself as its subalgebras and  $s=t=id_\fa$.
In this case, the kernel spaces $\Ker s$ and $\Ker t$ are  all zero, thus conditions \eqref{cat1}--\eqref{cat4} are natural satisfied.
\end{exe}

\begin{exe}\label{exe02}
Consider the 4-dimensional Hom-Lie antialgebra $M=\{\e;a_1, a_2, a_3\}$, where $\e$ is even and $a_1, a_2, a_3$ are odd,
and the linear map $(\a,\b):\fa\to \fa$ defined by
$$\a(\e)=\e, \quad \b(a_1)=\mu a_1,\quad \b(a_2)=\mu^{-1} a_2,\quad \b(a_3)=\mu a_3,$$
on the basis elements,  we obtain a Hom-Lie antialgebra structure given by
\begin{eqnarray*}
{[a_1,a_2]}=\e,\quad \e\cdot{}a_1=\mu a_3.
\end{eqnarray*}
Let $N=\{\e; a_1, a_2\}$ be its 3-dimensional subalgebra ${[a_1, a_2]}=\e$ with  projective map $\pi:M\to N$. There is an action of $N$ on $M$ in the natural way:
$$\rho_0(\e)(a_1)= \e\cdot{}a_1=\mu a_3,\quad\rho_1(a_1)(a_2)=[a_1,a_2]=\e,$$
and we obtain the semidirect product  $N\ltimes M$ with $N$ as its subalgebra by Theorem \ref{Thm1} in the above Section.
Now we define $s:N\ltimes M\to N$ and $t:N\ltimes M\to N$ by
\begin{eqnarray*}
&&s_0:N_0\ltimes M_0\to N_0, \quad (\e,\e')\mapsto \e\\
&&s_1:N_1\ltimes M_1\to N_1, \quad(a_i,a_j)\mapsto  a_i
\end{eqnarray*}
and
\begin{eqnarray*}
&&t_0:N_0\ltimes M_0\to N_0, \quad(\e,\e')\mapsto \e+\e'\\
&&t_1:N_1\ltimes M_1\to N_1,\quad (a_i,a_j)\mapsto  a_i+\pi(a_j)
\end{eqnarray*}
where $a_i\in N_1, i=1,2$ and $a_i\in M_1, i=1,2,3.$
Then we have $\Ker s_0=\{(0, \e)|\e\in M_0\}$ and $\Ker t_0=\{(\e,-\e)|\e\in M_0\}$ and
$\Ker s_1=\{(0, a_j)|a_j\in M_1\}$ and $\Ker t_1=\{(-\pi(a_j), a_j)|a_j\in M_1\}$.
It is easy to verify that $\Ker s_0\cdot \Ker t_0=0,\Ker s_0\cdot \Ker t_1=0$ and $\Ker t_0\cdot \Ker s_1=0, [\Ker s_1,\Ker t_1]=0.$
Therefore we obtain a $Cat^1$-Hom-Lie antialgebra with non-zero  kernel spaces $\Ker s$ and $\Ker t$.
\end{exe}

The construction in the above Example \ref{exe02} can be generalized to the following Theorem \ref{Thm02} which is the main result of this section.

\begin{thm}\label{Thm02}
The categories $\mathbf{XHLA}$ and $\mathbf{CHLA}$ are equivalent .
\end{thm}
\begin{proof} The proof is divided into two parts. Firstly, we define a functor from the category $\mathbf{XHLA}$ to $\mathbf{CHLA}$.
 Given a crossed module $\partial:(V ,\alpha_{V_0},\beta_{V_1})\rightarrow(\fa,\alpha,\beta)$, the corresponding $Cat^1$-Hom-Lie antiagebra $(\fa\ltimes V, \fa,s,t)$ are constructed as follows.
 The Hom-Lie antiagebras are the semidirect product $(\fa\ltimes V,\a+\a_{V_0},\b+\b_{V_1})$ with $(\fa,\alpha,\beta)$ as its Hom-Lie subantiagebra.

 The structural homomorphisms $s$ and $t$ are given by:
 $$s_0(x,u)=x,\quad t_0(x,u)=\partial_0(u)+x$$ and
 $$s_1(y,w)=y,\quad t_1(y,w)=\partial_1(w)+y$$
 for all $x\in\fa_0, ~y\in\fa_1, ~u\in V_0, ~w\in V_1$.

It is easy to check that $s$ is indeed a homomorphism since $s_0$ and $s_1$ are the projection maps.

For $t$, we have
\begin{eqnarray*}
t_0\circ(\alpha+\alpha_{V_0})(x,u)&=&t_0(\alpha(x),\alpha_{V_0}(u))=\partial_0(\alpha_{V_0}(u))+\alpha(x)\\
\alpha\circ t_0(x,u)&=&\alpha(\partial_0(u)+\alpha(x))=\alpha(\partial_0(u))+\alpha(x)
\end{eqnarray*}
\begin{eqnarray*}
t_1\circ(\beta+\beta_{V_1})(y,w)&=&t_1(\beta(y),\beta_{V_1}(w))=\partial_1(\beta_{V_1}(w))+\beta(y)\\
\beta\circ t_1(y,w)&=&\beta(\partial_1(w)+\beta(y))=\beta(\partial_1(w))+\beta(y).
\end{eqnarray*}
By computations, we get
\begin{eqnarray*}
t_0\big((x_1,u_1)\cdot(x_2,u_2)\big)&=&t_0(x_1\cdot x_2,~ u_1\cdot u_2+\rho_0(x_1)(u_2)+\rho_0(x_2)(u_1))\\
&=&\partial_0(u_1\cdot u_2)+\partial_0\circ\rho_0(x_1)(u_2)+\partial_0\circ\rho_0(x_2)(u_1)+x_1\cdot x_2\\
t_0(x_1,u_1)\cdot t_0(x_2,u_2)&=&(\partial_0(u_1)+x_1)\cdot(\partial_0(u_2)+x_2)\\
&=&\partial_0(u_1)\cdot\partial_0(u_2)+\partial_0(u_1)\cdot x_2+x_1\cdot\partial_0(u_2)+x_1\cdot x_2.
\end{eqnarray*}
Since $\partial=(\partial_0,\partial_1)$ is a Hom-Lie antialgebra homomorphism and by the crossed module conditions \eqref{cm1} and \eqref{cm2}, we get
\begin{eqnarray*}
t_0\big((x_1,u_1)\cdot(x_2,u_2)\big)=t_0(x_1,u_1)\cdot t_0(x_2,u_2).
\end{eqnarray*}
Similarly, by the crossed module conditions \eqref{cm3} and \eqref{cm4}, we get
\begin{eqnarray*}
t_1\big((x_1,u_1)\cdot(y_1,w_1)\big)&=&t_0(x_1,u_1)\cdot t_1(y_1,w_1),\\
t_0\big([((y_1,w_1),(y_2,w_2)]\big)&=&[t_1(y_1,w_1),t_1(y_2,w_2)].
\end{eqnarray*}
Therefore, $t$ is a Hom-Lie antiagebra homomorphism.

Now we check that $s$ and $t$ satisfy conditions \eqref{cat01}--\eqref{cat02} and \eqref{cat1}--\eqref{cat3}.
In fact $\fa$ can be regarded as a Hom-Lie subantialgebra of $\fa\ltimes V $ via the inclusion map $i_0:x\rightarrow(x,0)$ and $i_1:y\rightarrow(y,0)$, so it is obvious that $s_0|_{\fa_0}=t_0|_{\fa_0}=id_{\fa_0}$ and $s_1|_{\fa_1}=t_0|_{\fa_1}=id_{\fa_1}$. From the definition of $s$ and $t$, we get that $\Ker s_0=\{(0,u)|u\in V_0\}\cong V_0$, $\Ker t_0=\{(-\partial_0(u),u)|u\in V_0\}$, $\Ker s_1=\{(0,w)|w\in V_1\}\cong V_1$ and $\Ker t_1=\{(-\partial_1(w),w)|w\in V_1\}$). Given $u_1,~u_2\in V_0$, due to condition \eqref{pei1}, we have
\begin{eqnarray*}
(0,u_1)\cdot(-\partial_0(u_2),u_2)&=&(0,~u_1\cdot u_2-\rho_0(\partial_0(u_2))(u_1)=0.
\end{eqnarray*}
Therefore, we get $\Ker s_0\cdot \Ker t_0=0$. One can also obtain $\Ker s_0\cdot \Ker t_1=0$, $\Ker t_0\cdot \Ker s_1=0$ and $[\Ker s_1,\Ker t_1]=0$ by conditions \eqref{pei2} and \eqref{pei3}.

Moreover, given a morphism $(\phi,\psi)$ of crossed modules, the corresponding morphism of $Cat^1$-Hom-Lie antialgebras is defined by $f_0(x,u)\triangleq(\psi_0(x),\phi_0(u))$, $f_1(y,w)\triangleq(\psi_1(y),\phi_1(w))$ for all $(x,u)\in(\fa\ltimes V)_0 ,~(y,w)\in(\fa\ltimes V)_1 $. One can check that $(f_0,f_1)$ is a homomorphisms of Hom-Lie antialgebras since the pair $(\phi,\psi)$ is a morphism of Hom-Lie antialgebra crossed modules. For example,
\begin{eqnarray*}
  &&f_0\big((x_1,u_1)\cdot(x_2,u_2)\big)\\
  &=& f_0(x_1\cdot x_2,~ u_1\cdot u_2+\rho_0(x_1)(u_2)+\rho_0(x_2)(u_1)\\
  &=&(\psi_0(x_1\cdot x_2),~\phi_0(u_1\cdot u_2+\rho_0(x_1)(u_2)+\rho_0(x_2)(u_1))\\
  &=&(\psi_0(x_1\cdot x_2),~\phi_0(u_1\cdot u_2)+\rho_0(\psi_0(x_1))(\phi_0(u_2))+\rho_0(\psi_0(x_2))(\phi_0(u_1))\\
  &=&(\psi_0(x_1),\phi_0(u_1))\cdot(\psi_0(x_2),\phi_0(u_2))\\
  &=&f_0(x_1,u_1)\cdot f_0(x_2,u_2).
  \end{eqnarray*}
Thus $f_0\big((x_1,u_1)\cdot(x_2,u_2)\big)=f_0(x_1,u_1)\cdot f_0(x_2,u_2)$.
It is clear that $f_0(\fa_0)\subseteq\widetilde{\fa_0},~f_1(\fa_1)\subseteq\widetilde{\fa_1}$. The equalities $\widetilde{s_0}\circ f_0=f_0|_{\fa_0}\circ s_0$, $\widetilde{s_1}\circ f_1=f_1|_{\fa_1}\circ s_1$ follow from the definition of $s_0,~s_1$. Moreover, the equalities $\widetilde{t_0}\circ f_0=f_0|_{\fa_0}\circ t_0$, $\widetilde{t_1}\circ f_1=f_1|_{\fa_1}\circ t_1$ are  consequence of the fact that $s$ and $t$ are Hom-Lie antialgebra homomorphisms. Therefore the functor from $\mathbf{XHLA}$ to $\mathbf{CHLA}$ is well defined.

For the second part, we define a functor from the category $\mathbf{CHLA}$ to $\mathbf{XHLA}$ as follows. Given a $Cat^1$-Hom-Lie antialgebra $\left(M, N,s,t\right)$, we define the corresponding crossed module of Hom-Lie antialgebras to be
 $t|_{\Ker s}:\Ker s \to N$ where $t_0|_{\Ker s_0}:\Ker s_0\rightarrow N_0$, $t_1|_{\Ker s_1}:\Ker s_1\rightarrow N_1$, with actions of $N$ on $\Ker s$ induced by the Hom-Lie antialgebra structure in $M$.  By the fact that $s_0(x\cdot p)=s_0(x)\cdot s_0(p)=0$, $s_1(x\cdot q)=s_0(x)\cdot s_1(q)=0$, $s_1(y\cdot p)=s_1(y)\cdot s_0(p)=0$ and $s_0([y, q])=[s_1(y), s_1(q)]=0$, we define
\begin{eqnarray*}
\rho_0(x)(p)&=&x\cdot p\in \Ker s_0,\quad \rho_1(y)(q)=[y, q]\in \Ker s_0,\\
\rho_0(x)(q)&=&x\cdot q\in \Ker s_1,\quad \rho_1(y)(p)=y\cdot p\in \Ker s_1.
\end{eqnarray*}
where $x\in N_0, y\in N_1, p\in \Ker s_0, q\in \Ker s_1$.
 Then $t|_{\Ker s}=(t_0|_{\Ker s_0},t_1|_{\Ker s_1})$ satisfy the crossed module conditions \eqref{cm1}--\eqref{cm4} since $t$ is a Hom-Lie antialgebra homomorphism and $t_0|_{N_0}=id_{N_0}, t_1|_{N_1}=id_{N_1}$.

For the crossed module conditions \eqref{pei1}--\eqref{pei3}, assume $p_1\in \Ker s_0$, $q_1\in \Ker s_1$, we have $t_0 (t_0 (p_1)-p_1)=t_0 (p_1)-t_0 (p_1)=0$ and $t_1 (t_1 (q_1)-q_1)=t_1 (q_1)-t_1 (q_1)=0$ since $t_0 (p_1)\in N_0$ and $t_1 (q_1)\in N_1$. Thus $t_0 (p_1)-p_1\in \Ker t_0 $ and $t_1 (q_1)-q_1\in \Ker t_1$. Thus we have
\begin{eqnarray*}
0&=&(t_0 (p_1)-p_1)\cdot p_2=t_0 (p_1)\cdot p_2-p_1\cdot p_2,\\
0&=&(t_0 (p_1)-p_1)\cdot q_1=t_0 (p_1)\cdot q_1-p_1\cdot q_1,\\
0&=&(t_1 (q_1)-q_1)\cdot p_1=t_1 (q_1)\cdot p_1-q_1\cdot p_1,\\
0&=&[t_1 (q_1)-q_1,q_2]=[t_1 (q_1),q_2]-[q_1,q_2].
\end{eqnarray*}
Therefore $t|_{\Ker s}:\Ker s \to N$ is a crossed module.

Moreover, given a morphism $f=(f_0,f_1)$ of $Cat^1$-Hom-Lie antialgebras, the corresponding morphism of  crossed modules is given by $(f|_{\Ker_{s}},~f|_{N})$. Since $\widetilde{s_0}\circ f_0=f_0|_{N_0}\circ s_0$ and $\widetilde{s_1}\circ f_1=f_1|_{N_1}\circ s_1$, we obtain that $f_0(\Ker s_0)\subseteq \Ker \widetilde{s_0}$ and $f_1(\Ker s_1)\subseteq \Ker \widetilde{s_1}$.  directly from the identities $\widetilde{t_0}\circ f_0=f_0|_{N_0}\circ t_0$ and $\widetilde{t_1}\circ f_1=f_1|_{N_1}\circ t_1$.
Thus the functor from $\mathbf{CHLA}$ to $\mathbf{XHLA}$ is well  defined.

Finally, one can see that the above two functors are inverse of each other up to isomorphim.
This completes the proof.
\end{proof}

\section{Crossed module extensions and third cohomology groups}

The cohomology theory of Hom-Lie antialgebras was defined in \cite{ZZ}.
We will define the notion of crossed module extensions of Hom-Lie antialgebras and show that they are related to the third cohomology group.

\begin{defi}
 Let $(\fg,\alpha,\beta)$ be a Hom-Lie antialgebra and $(M,\r)$ be a representation of $(\fg,\alpha,\beta)$.
 A crossed module extension of $(\fg,\alpha,\beta)$ by $(M,\r)$ is an exact sequence of Hom-Lie antialgebras:
\begin{equation}\label{exact1}
\xymatrix{
  0 \ar[r] & M  \ar[r]^{i} & V  \ar[r]^{\partial } & \fa  \ar[r]^{\pi } & \fg  \ar[r] & 0 },
\end{equation}
such that $\partial:V\rightarrow\fa$ is a crossed module and $M_0\cong \Ker\partial_0,~M_1\cong \Ker \partial_1$, $\fg_0\cong \mathrm{coKer} \partial_0,~\fg_1\cong \mathrm{coKer} \partial_1$.
\end{defi}

\begin{defi}
Two crossed module extensions $\partial:V\rightarrow\fa$  and $\widetilde{\partial}:\widetilde{V}\rightarrow\widetilde{\fa}$   are called equivalent if there are homomorphisms of Hom-Lie antialgebras $\phi:V\rightarrow\widetilde{V}$ and $\psi:\fa\rightarrow\widetilde{\fa}$ which are compatible with the actions such that the following diagram commutate:
 \begin{equation}
\xymatrix{
    0\ar[r]^{} & M  \ar@{=}[d]_{} \ar[r]^{i} & V  \ar[d]_{\phi } \ar[r]^{\partial } & \fa  \ar[d]_{\psi}  \ar[r]^{\pi } & \fg  \ar@{=}[d]_{} \ar[r]^{} & 0  \\
   0\ar[r]^{} & M  \ar[r]^{i'} & \widetilde{V}  \ar[r]^{\widetilde{\partial} } & \widetilde{\fa}  \ar[r]^{\pi' } & \fg  \ar[r]^{} & 0.}
   \end{equation}
\end{defi}
In other words, $\partial$ and $\widetilde{\partial}$ are equivalent if there are Hom-Lie antialgebra homomorphisms $\phi_0:V_0\rightarrow\widetilde{V}_0,~\phi_1:V_1\rightarrow\widetilde{V}_1$, such that $\widetilde{\partial}_0\circ\phi_0=\partial_0\circ\psi_0,
~\widetilde{\partial}_1\circ\phi_1=\partial_1\circ\psi_1$.
Let $\mathbf{CExt}(\fg,M)$ denote the set of equivalence classes of the set of crossed module extensions.

\begin{thm}
 Let $(\fg,\alpha,\beta)$ be a Hom-Lie antialgebra and $(M,\r)$ be a representation of $(\fg,\alpha,\beta)$.
  Then there is a canonical map:
\begin{equation*}
  \xi:\mathbf{CExt}(\fg,M)\rightarrow \mathbf{H^3}(\fg,M)
\end{equation*}
\end{thm}

\begin{proof}
Given a crossed module extension  of Hom-Lie antialgebra, we denote $\fn_0=\Ker\pi_0=\mathrm{Im}\partial_0$, $\fn_1=\Ker\pi_1=\mathrm{Im}\partial_1$. Choose linear sections $s_0:\fg_0\rightarrow\fa_0,~s_1:\fg_1\rightarrow\fa_1$ and $\sigma_0:\fn_0\rightarrow V_0,~\sigma_1:\fn_1\rightarrow V_1$ of $\pi_0,~\pi_1$ and $\partial_0,~\partial_1$  such that $\pi_0 s_0=\mathrm{id}_{\fg_0}$, $\pi_1 s_1=\mathrm{id}_{\fg_1}$, $\partial_0\sigma_0=\mathrm{id}_{\fn_0}$ and $\partial_1\sigma_1=\mathrm{id}_{\fn_1}$. Since $\pi$ is a Hom-Lie antialgebra homomorphism, then forall $a_1,a_2,a_3\in\fg_0,~b_1,b_2,b_3\in\fg_1$, we get
\begin{eqnarray*}
&&\pi_0\big(s_0(a_1)\cdot s_0(a_2)-s_0(a_1\cdot a_2)\big)=\pi_0s_0(a_1)\cdot \pi_0s_0(a_2)-\pi_0s_0(a_1\cdot a_2)=0,\\
&&\pi_1\big(s_0(a_1)\cdot s_1(b_1)-s_1(a_1\cdot b_1)\big)=\pi_0s_0(a_1)\cdot \pi_1s_1(b_1)-\pi_1s_1(a_1\cdot b_1)=0,\\
&&\pi_0\big([s_1(b_1), s_1(b_2)]-s_0([b_1, b_2])\big)=[\pi_1s_1(b_1), \pi_1s_1(b_2)]-\pi_0s_0([b_1, b_2])=0.
\end{eqnarray*}
Thus we have $s_0(a_1)\cdot s_0(a_2)-s_0(a_1\cdot a_2)\in \fn_0$, $s_0(a_1)\cdot s_1(b_1)-s_1(a_1\cdot b_1)\in \fn_1$ and $[s_1(b_1), s_1(b_2)]-s_0([b_1, b_2])\in \fn_0$.

Take $\omega_0:\fg_0\times \fg_0\rightarrow V_0,~\omega_1:\fg_0\times \fg_1\rightarrow V_1,~\omega_2:\fg_1\times \fg_1\rightarrow V_0$ as follows:
\begin{eqnarray*}
\omega_0(a_1,a_2)&=&\sigma_0\big(s_0(a_1)\cdot s_0(a_2)-s_0(a_1\cdot a_2)\big),\\
\omega_1(a_1,b_1)&=&\sigma_1\big(s_0(a_1)\cdot s_1(b_1)-s_1(a_1\cdot b_1)\big),\\
\omega_2(b_1,b_2)&=&\sigma_0\big([s_1(b_1), s_1(b_2)]-s_0([b_1, b_2])\big),
\end{eqnarray*}
for all $a_1,a_2\in \fg_0, b_1,b_2\in \fg_1$.
We define maps
\begin{eqnarray*}
 h_0(a_1,a_2,a_3)&=&\r_0(s_0(\alpha(a_1)))\omega_0(a_2,a_3)+\omega_0(\alpha(a_1),a_2\cdot a_3)\\
&&-\r_0(s_0(\alpha(a_3)))\omega_0(a_1,a_2)-\omega_0(a_1\cdot a_2,\alpha(a_3)),
\end{eqnarray*}
\begin{eqnarray*}
h_1(a_1,a_2,b_1)&=&\r_0(s_0(\alpha(a_1)))\omega_1(a_2,b_1)+\omega_1(\alpha(a_1),a_2\cdot b_1)\\
&&-\half\r_1(s_1(\beta(b_1)))\omega_0(a_1,a_2)-\half\omega_1(a_1\cdot a_2,\beta(b_1)),
\end{eqnarray*}
\begin{eqnarray*}
h_2(a_1,b_1,b_2)&=&\r_0(s_0(\alpha(a_1)))\omega_2(b_1,b_2)+\omega_0(\alpha(a_1),[b_1,b_2])\\
&&-\r_1(s_1(\beta(b_2)))\omega_1(a_1,b_1)-\omega_2(a_1\cdot b_1,\beta(b_2))\notag\\
 &&-\r_1(s_1(\beta(b_1)))\omega_1(a_1,b_2)-\omega_2(\beta(b_1),a_1\cdot b_2),
\end{eqnarray*}
\begin{eqnarray*}
h_3(b_1,b_2,b_3)&=&\r_1(s_1(\beta(b_1)))\omega_2(b_2,b_3)+\omega_1(\beta(b_1),[b_2,b_3])\\
&&+\r_1(s_1(\beta(b_2)))\omega_2(b_3,b_1)+\omega_1(\beta(b_2),[b_3,b_1])\\
&&+\r_1(s_1(\beta(b_3)))\omega_2(b_1,b_2)+\omega_1(\beta(b_3),[b_1,b_2]).
\end{eqnarray*}

Now check that the image of those maps are contained in the kernel of $\partial$.
By definition we have
\begin{eqnarray*}
   &&\partial_0h_0(a_1,a_2,a_3)\\
   &=&\partial_0\big(\r_0(s_0(\alpha(a_1)))\omega_0(a_2,a_3)\big)+\partial_0(\omega_0(\alpha(a_1),a_2\cdot a_3))\\
   &&-\partial_0(\r_0(s_0(\alpha(a_3)))\omega_0(a_1,a_2))-\partial_0(\omega_0(a_1\cdot a_2,\alpha(a_3)))\\
   &=&\partial_0\big\{\r_0(s_0(\alpha(a_1)))\sigma_0(s_0(a_2)\cdot s_0(a_3)-s_0(a_2\cdot a_3))\big\}\\
   &&+\partial_0\big\{\sigma_0(s_0(\alpha(a_1))\cdot s_0(a_2\cdot a_3)-s_0(\alpha(a_1))\cdot(a_2\cdot a_3))\big\}\\
   &&-\partial_0\big\{\r_0(s_0(\alpha(a_3)))\sigma_0(s_0(a_1)\cdot s_0(a_2)-s_0(a_1\cdot a_2))\big\}\\
   &&-\partial_0\big\{\sigma_0(s_0(a_1\cdot a_2)\cdot s_0(\alpha(a_3))-s_0((a_1\cdot a_2)\cdot(\alpha(a_3)))\big\}
\end{eqnarray*}
due to \eqref{cm1}, we get
\begin{eqnarray*}
  &&\partial_0h_0(a_1,a_2,a_3) \\
  &=& s_0(\alpha(a_1))\cdot (s_0(a_2)\cdot s_0(a_3))-s_0(\alpha(a_1))\cdot s_0(a_2\cdot a_3)\\
  &&+s_0(\alpha(a_1))\cdot s_0(a_2\cdot a_3)-s_0(\alpha(a_1)\cdot(a_2\cdot a_3))\\
  &&-s_0(\alpha(a_3))\cdot (s_0(a_1)\cdot s_0(a_2))+s_0(\alpha(a_3))\cdot s_0(a_1\cdot a_2)\\
  &&-s_0(a_1\cdot a_2)\cdot s_0(\alpha(a_3))+s_0((a_1\cdot a_2)\cdot\alpha(a_3))\\
  &=& s_0(\alpha(a_1))\cdot (s_0(a_2)\cdot s_0(a_3))-s_0(\alpha(a_1)\cdot(a_2\cdot a_3))\\
  &&- s_0(\alpha(a_3))\cdot (s_0(a_1)\cdot s_0(a_2))+s_0((a_1\cdot a_2)\cdot\alpha(a_3))\\
  &=&0.
\end{eqnarray*}
Thus $h_0(a_1,a_2,a_3)\in M_0=\Ker(\partial_0)$.

Similarly, due to \eqref{cm2},\eqref{cm3}, we have
\begin{eqnarray*}
  &&\partial_1h_1(a_1,a_2,b_1) \\
  &=& s_0(\alpha(a_1))\cdot(s_0(a_2)\cdot s_1(b_1))-s_1(\alpha(a_1)\cdot(a_2\cdot b_1))\\
  &&-\half s_1(\beta(b_1))\cdot(s_0(a_1)\cdot s_0(a_2))+\half s_1((a_1\cdot a_2)\cdot\beta(b_1))\\
  &=&0.
\end{eqnarray*}
Due to \eqref{cm1},\eqref{cm4}, we have
\begin{eqnarray*}
  &&\partial_0h_2(a_1,b_1,b_2) \\
  &=& s_0(\alpha(a_1))\cdot [s_1(b_1), s_1(b_2)]-s_0(\alpha(a_1)\cdot[b_1, b_2])\\
  &&- [s_1(\beta(b_2)),s_0(a_1)\cdot s_1(b_1)]+s_0[a_1\cdot b_1,\beta(b_2)]\\
  &&- [s_1(\beta(b_1)),s_0(a_1)\cdot s_1(b_2)]+s_0[\beta(b_1),a_1\cdot b_2]\\
  &=&0
\end{eqnarray*}
and
\begin{eqnarray*}
  &&\partial_1h_3(b_1,b_2,b_3) \\
  &=& s_1(\beta(b_1))\cdot [s_1(b_2), s_1(b_3)]-s_1(\beta(b_1)\cdot[b_2, b_3])\\
  &&+ s_1(\beta(b_2))\cdot [s_1(b_3), s_1(b_1)]-s_1(\beta(b_2)\cdot[b_3, b_1])\\
  &&+ s_1(\beta(b_3))\cdot [s_1(b_1), s_1(b_2)]-s_1(\beta(b_3)\cdot[b_1, b_2])\\
  &=&0.
\end{eqnarray*}
 Therefore we have $h_2(a_1,b_1,b_2)\in M_0$ $=\Ker(\partial_0)$ and $h_1(a_1,a_2,b_1)$, $h_3(b_1,b_2,b_3)\in M_1$$=\Ker(\partial_1)$. Thus we defined a map $h:\wedge^3\fg\rightarrow M$, that is to say $h\in C^3(\fg,M)$. Routine but complicated calculations show that $d(h)=0$. Hence the map $h$ defines a 3-cocycle in the cohomology of $\fg$ with coefficients in $M$.

 Now we are going to check that $\xi$ is a well-defined map, i.e. the equivalent class of $h$ does not depend on the sections $s,\sigma$.

 First we show that the class of $h$ does not depend on the sections $s$. Suppose $\overline{s}:\fg\rightarrow\fa$ is another section of $\pi$ and let $\overline{h}$ be the corresponding 3-cocycle defined using $\overline{s}$ instead of $s$. Then there exist linear maps $f_0:\fg_0\rightarrow V_0$ and $f_1:\fg_1\rightarrow V_1$ with $s-\overline{s}=\partial f$.

Due to \eqref{pei1},we have
 \begin{eqnarray*}
&&\r_0(s_0(\alpha(a_1)))\omega_0(a_2,a_3)-\r_0(\overline{s}_0(\alpha(a_1)))\overline{\omega}_0(a_2,a_3)\\
&=&\r_0(s_0(\alpha(a_1)))\omega_0(a_2,a_3)-\r_0(\overline{s}_0(\alpha(a_1)))\omega_0(a_2,a_3)\\
&&+\r_0(\overline{s}_0(\alpha(a_1)))\omega_0(a_2,a_3)-\r_0(\overline{s}_0(\alpha(a_1)))\overline{\omega}_0(a_2,a_3)\\
&=&\r_0((s_0-\overline{s}_0)(\alpha(a_1)))\omega_0(a_2,a_3)+
\r_0(\overline{s}_0(\alpha(a_1)))\big((\omega_0-\overline{\omega}_0)(a_2,a_3)\big)\\
&=&\r_0(\partial_0f_0(\alpha(a_1)))\omega_0(a_2,a_3)+
\r_0(\overline{s}_0(\alpha(a_1)))\big((\omega_0-\overline{\omega}_0)(a_2,a_3)\big)\\
&=&f_0(\alpha(a_1))\cdot\big((s_0(a_2)\cdot s_0(a_3)-s_0(a_2\cdot a_3))\big)\\
&&+\r_0(\overline{s}_0(\alpha(a_1)))\big((\omega_0-\overline{\omega}_0)(a_2,a_3)\big),
 \end{eqnarray*}
then
\begin{eqnarray*}
&&(h_0-\overline{h}_0)(a_1,a_2,a_3)\\
&=& f_0(\alpha(a_1))\cdot\big((s_0(a_2)\cdot s_0(a_3)-s_0(a_2\cdot a_3))\big)\\
&&+\r_0(\overline{s}_0(\alpha(a_1)))\big((\omega_0-\overline{\omega}_0)(a_2,a_3)\big)\\
&&-f_0(\alpha(a_3))\cdot\big((s_0(a_1)\cdot s_0(a_2)-s_0(a_1\cdot a_2))\big)\\
&&-\r_0(\overline{s}_0(\alpha(a_3)))\big((\omega_0-\overline{\omega}_0)(a_1,a_2)\big)\\
&&+(\omega_0-\overline{\omega}_0)(\alpha(a_1),a_2\cdot a_3)-(\omega_0-\overline{\omega}_0)(a_1\cdot a_2,\alpha(a_3)).
\end{eqnarray*}
Due to \eqref{pei3}, we have
\begin{eqnarray*}
&&(h_1-\overline{h}_1)(a_1,a_2,b_1)\\
&=& f_0(\alpha(a_1))\cdot\big((s_0(a_2)\cdot s_1(b_1)-s_1(a_2\cdot b_1))\big)\\
&&+\r_0(\overline{s}_0(\alpha(a_1)))\big((\omega_1-\overline{\omega}_1)(a_2,b_1)\big)\\
&&-\half f_1(\beta(b_1))\cdot\big((s_0(a_1)\cdot s_0(a_2)-s_0(a_1\cdot a_2))\big)\\
&&-\half \r_1(\overline{s}_1(\beta(b_1)))\big((\omega_0-\overline{\omega}_0)(a_1,a_2)\big)\\
&&+(\omega_1-\overline{\omega}_1)(\alpha(a_1),a_2\cdot b_1)-\half(\omega_1-\overline{\omega}_1)(a_1\cdot a_2,\beta(b_1)).
\end{eqnarray*}
Due to \eqref{pei2} and \eqref{pei3}, we have
\begin{eqnarray*}
&&(h_2-\overline{h}_2)(a_1,b_1,b_2)\\
&=& f_0(\alpha(a_1))\cdot\big([s_1(b_1), s_1(b_2)]-s_0[b_1, b_2]\big)\\
&&+\r_0(\overline{s}_0(\alpha(a_1)))\big((\omega_2-\overline{\omega}_2)(b_1,b_2)\big)\\
&&-[f_1(\beta(b_2)),s_0(a_1)\cdot s_1(b_1)-s_1(a_1\cdot b_1)]\\
&&-\r_1(\overline{s}_1(\beta(b_1)))\big((\omega_1-\overline{\omega}_1)(a_1,b_1)\big)\\
&&-[f_1(\beta(b_1)),s_0(a_1)\cdot s_1(b_2)-s_1(a_1\cdot b_2)]\\
&&-\r_1(\overline{s}_1(\beta(b_1)))\big((\omega_1-\overline{\omega}_1)(a_1,b_2)\big)\\
&&+(\omega_0-\overline{\omega}_0)(\alpha(a_1),[b_1, b_2])-(\omega_2-\overline{\omega}_2)(a_1\cdot b_1,\beta(b_2))\\
&&-(\omega_2-\overline{\omega}_2)(\beta(b_1),a_1\cdot b_2).
\end{eqnarray*}
Due to \eqref{pei2}, we have
\begin{eqnarray*}
  &&(h_3-\overline{h}_3)(b_1,b_2,b_3)\\
&=& f_1(\beta(b_1))\cdot\big([s_1(b_2), s_1(b_3)]-s_0[b_2, b_3]\big)\\
&&+\r_1(\overline{s}_1(\beta(b_1)))\big((\omega_2-\overline{\omega}_2)(b_2,b_3)\big)\\
&&+f_1(\beta(b_2))\cdot\big([s_1(b_3), s_1(b_1)]-s_0[b_3, b_1]\big)\\
&&+\r_1(\overline{s}_1(\beta(b_2)))\big((\omega_2-\overline{\omega}_2)(b_3,b_1)\big)\\
&&+f_1(\beta(b_3))\cdot\big([s_1(b_1), s_1(b_2)]-s_0[b_1, b_2]\big)\\
&&+\r_1(\overline{s}_1(\beta(b_3)))\big((\omega_2-\overline{\omega}_2)(b_1,b_2)\big)\\
&&+(\omega_1-\overline{\omega}_1)(\beta(b_1),[b_2, b_3])+(\omega_1-\overline{\omega}_1)(\beta(b_2), [b_3, b_1])\\
&&+(\omega_1-\overline{\omega}_1)(\beta(b_3),[b_1, b_2]).
\end{eqnarray*}
Next we define maps $\lambda:\wedge^2\fg\rightarrow V$ by
\begin{eqnarray*}
 \lambda_0(a_1,a_2)&=&\r_0(\overline{s}_0(a_1))f_0(a_2)+\r_0(\overline{s}_0(a_2))f_0(a_1)
 -f_0(a_1)\cdot (\partial_0f_0)(a_2)-f_0(a_1\cdot a_2)\\
  \lambda_1(a_1,b_1)&=&\r_0(\overline{s}_0(a_1))f_1(b_1)+\r_1(\overline{s}_1(b_1))f_0(a_1)
 -f_0(a_1)\cdot (\partial_1f_1)(b_1)-f_1(a_1\cdot b_1)\\
 \lambda_2(b_1,b_2)&=&\r_1(\overline{s}_1(b_1))f_1(b_2)+\r_1(\overline{s}_1(b_2))f_1(b_1)
 -f_1(b_1)\cdot (\partial_1f_1)(b_2)-f_0(b_1\cdot b_2)
\end{eqnarray*}
Then an easy caculation shows that  $\partial\lambda=\partial(\omega-\overline{\omega})$, hence $(\omega-\overline{\omega}-\lambda):\wedge^2\fg\rightarrow M$.

Moreover, if we replace $\omega_0-\overline{\omega}_0$ by $\lambda_0$, then we have
\begin{eqnarray}\label{app}
(h_0-\overline{h}_0)(a_1,a_2,a_3)=d (\omega_0-\overline{\omega}_0-\lambda_0)(a_1,a_2,a_3).
\end{eqnarray}
See Appendix at the end of this paper for the detailed proof of equation \eqref{app}.

Similarly, if we replace $\omega_0-\overline{\omega}_0$ by $\lambda_0,~\omega_1-\overline{\omega}_1$ by $\lambda_1$, then we get
$  (h_1-\overline{h}_1)=d (\omega_0-\overline{\omega}_0-\lambda_0)+d (\omega_1-\overline{\omega}_1-\lambda_1).$

If we replace $\omega_0-\overline{\omega}_0$ by $\lambda_0,~\omega_1-\overline{\omega}_1$ by $\lambda_1,~\omega_2-\overline{\omega}_2$ by $\lambda_2$, then we get
$(h_2-\overline{h}_2)=
d (\omega_0-\overline{\omega}_0-\lambda_0)+d (\omega_1-\overline{\omega}_1-\lambda_1)+d (\omega_2-\overline{\omega}_2-\lambda_2).$

If we replace $\omega_1-\overline{\omega}_1$ by $\lambda_1,~\omega_2-\overline{\omega}_2$ by $\lambda_2$, then we get
$(h_3-\overline{h}_3)=d (\omega_1-\overline{\omega}_1-\lambda_1)+d (\omega_2-\overline{\omega}_2-\lambda_2).$

Therefore, we obtain $h-\overline{h}=d(\omega-\overline{\omega}-\lambda)$.
Hence the class of $h$ does not depend on the section $s$.

Next we consider the map
 \begin{equation*}
\xymatrix{
    0\ar[r]^{} & M  \ar@{=}[d]_{} \ar[r]^{i} & V  \ar[d]_{\phi } \ar[r]^{\partial } & \fa  \ar[d]_{\psi}  \ar[r]^{\pi } & \fg  \ar@{=}[d]_{} \ar[r]^{} & 0  \\
   0\ar[r]^{} & M  \ar[r]^{\widetilde{i}} & \widetilde{V}  \ar[r]^{\widetilde{\partial} } & \widetilde{\fa}  \ar[r]^{\widetilde{\pi} } & \fg  \ar[r]^{} & 0,   }
   \end{equation*}
of crossed modules.

Let $\widetilde{s}:\fg\rightarrow\widetilde{\fa}$ and $\widetilde{\sigma}:\widetilde{\fn}\rightarrow \widetilde{V}$ be section of $\widetilde{\pi}$ and $\widetilde{\partial}$, respectively. Note that $(\widetilde{\pi}_0\psi_0 s_0)(a)=(\pi_0 s_0)(a)=a,~(\widetilde{\pi}_1\psi_1 s_1)(b)=(\pi_1 s_1)(b)=b$, for all $a\in\fg_0,~b\in\fg_1$. Therefore, $\psi s:\fg\rightarrow\widetilde{\fa}$ is another section of $\widetilde{\pi}$, then choose another section $\widetilde{\sigma}$, we have
\begin{eqnarray*}
  &&h_0(a_1,a_2,a_3)-\widetilde{h}_0(a_1,a_2,a_3)\\
   &=& \r_0(s_0(\alpha(a_1)))\omega_0(a_2,a_3)+\omega_0(\alpha(a_1),a_2\cdot a_3)\\
&&-\r_0(s_0(\alpha(a_3)))\omega_0(a_1,a_2)-\omega_0(a_1\cdot a_2,\alpha(a_3)),\\
&&-\r_0(\psi_0s_0(\alpha(a_1)))\widetilde{\omega}_0(a_2,a_3)-\widetilde{\omega}_0(\alpha(a_1),a_2\cdot a_3)\\
&&+\r_0(\psi_0s_0(\alpha(a_3)))\widetilde{\omega}_0(a_1,a_2)+\widetilde{\omega}_0(a_1\cdot a_2,\alpha(a_3))\\
&=&\r_0(\psi_0s_0(\alpha(a_1)))\big((\phi_0\sigma_0-\widetilde{\sigma}_0\psi_0)(s_0(a_2)\cdot s_0(a_3)-s_0(a_2\cdot a_3))\big)\\
&&-\r_0(\psi_0s_0(\alpha(a_3)))\big((\phi_0\sigma_0-\widetilde{\sigma}_0\psi_0)(s_0(a_1)\cdot s_0(a_2)-s_0(a_1\cdot a_2))\big)\\
&&+(\phi_0\sigma_0-\widetilde{\sigma}_0\psi_0)\big(s_0(\alpha(a_1))\cdot s_0(a_2\cdot a_3)-s_0(\alpha(a_1)\cdot(a_2\cdot a_3))\big)\\
&&-(\phi_0\sigma_0-\widetilde{\sigma}_0\psi_0)\big(s_0(a_1\cdot a_2)\cdot s_0(\alpha(a_3))-s_0((a_1\cdot a_2)\cdot\alpha(a_3))\big)\\
&=&(d^2\theta_0)(a_1,a_2,a_3),
\end{eqnarray*}
where
\begin{eqnarray*}
\theta_0(a_1,a_2) &=& (\phi_0\sigma_0-\widetilde{\sigma}_0\psi_0)(s_0(a_1)\cdot s_0(a_2)-s_0(a_1\cdot a_2)).
\end{eqnarray*}
One can also check that
$h_1-\widetilde{h}_1=d\theta_0+d\theta_1$, $h_2-\widetilde{h}_2=d\theta_0+d\theta_1+d\theta_2$,
$h_3-\widetilde{h}_3=d\theta_1+d\theta_2$,
where
\begin{eqnarray*}
  \theta_1(a_1,b_1) &=& (\phi_1\sigma_1-\widetilde{\sigma}_1\psi_1)(s_0(a_1)\cdot s_1(b_1)-s_0(a_1\cdot b_1)),\\
  \theta_2(b_1,b_2) &=& (\phi_0\sigma_0-\widetilde{\sigma}_0\psi_0)([s_1(b_1), s_1(b_2)]-s_0[b_1, b_2]).
\end{eqnarray*}
This prove that $h=\widetilde{h}$ in $\mathbf{H}^3(\fg,M)$ and the class of $h$ does not depend on the section $\sigma$. Therefore the map $\xi$ is well defined.
\end{proof}

One would like to establish a bijection between $\mathbf{CExt}(\fg,M)$ and $\mathbf{H^3}(\fg,M)$.
For this, we have to construct a canonical example of crossed extension for a given cohomology class.
We don't have obtained this result yet.
The main  obstacle is that we don't know if category of Hom-Lie antialgebra representations possesses enough injective objects, or the higher cohomology classes of a free Hom-Lie antialgebra are trivial. These problems are left for future investigations.

\appendix
\section{Proof of  equation \eqref{app}}
First we replace $\omega_0-\overline{\omega}_0$ by $\lambda_0$, then
\begin{eqnarray*}
&&(h_0-\overline{h}_0)(a_1,a_2,a_3)\\
&=& f_0(\alpha(a_1))\cdot(s_0(a_2)\cdot s_0(a_3))\underline{A_1}- f_0(\alpha(a_1))\cdot s_0(a_2\cdot a_3)\underline{A_2}\\
&&+\r_0(\overline{s}_0(\alpha(a_1)))\big(\lambda_0(a_2,a_3)\big)\\
&&-f_0(\alpha(a_3))\cdot(s_0(a_1)\cdot s_0(a_2))\underline{A_3}+f_0(\alpha(a_3))\cdot s_0(a_1\cdot a_2)\underline{A_4}\\
&&-\r_0(\overline{s}_0(\alpha(a_3)))\big(\lambda_0(a_1,a_2)\big)\\
&&+\lambda_0(\alpha(a_1),a_2\cdot a_3)-\lambda_0(a_1\cdot a_2,\alpha(a_3))\\
&&+(d (\omega_0-\overline{\omega}_0-\lambda_0))(a_1,a_2,a_3).
\end{eqnarray*}
Second we substitute the definition of $\lambda$ and $\overline{s_0}$ in above equation, we obtain
\begin{eqnarray*}
&&\r_0(\overline{s}_0(\alpha(a_1)))\big(\lambda_0(a_2,a_3)\big)\\
&=&\r_0((s_0-\partial_0f_0)(\alpha(a_1)))\big\{\r_0((s_0-\partial_0f_0)(a_2))f_0(a_3)\\
&&+\r_0((s_0-\partial_0f_0)(a_3))f_0(a_2)
-f_0(a_2)\cdot (\partial_0f_0)(a_3)-f_0(a_2\cdot a_3)\big\}\\
&=&\r_0(s_0(\alpha(a_1)))(\r_0(s_0(a_2))f_0(a_3))\underline{B_1}-\r_0(s_0(\alpha(a_1)))\r_0((\partial_0f_0)(a_2))f_0(a_3)\underline{B_2}\\
&&+\r_0(s_0(\alpha(a_1)))(\r_0(s_0(a_3))f_0(a_2))\underline{B_3}-\r_0(s_0(\alpha(a_1)))\r_0((\partial_0f_0)(a_3))f_0(a_2)\underline{B_4}\\
&&-\r_0(s_0(\alpha(a_1)))((f_0(a_2)\cdot \partial_0f_0(a_3))\underline{B_5}-\r_0(s_0(\alpha(a_1)))f_0(a_2\cdot a_3)\underline{B_6}\\
&&-\r_0((\partial_0f_0)(\alpha(a_1)))(\r_0(s_0(a_2))f_0(a_3))\underline{B_7}\\
&&+\r_0((\partial_0f_0)(\alpha(a_1)))\r_0((\partial_0f_0)(a_2))f_0(a_3)\underline{B_8}\\
&&-\r_0((\partial_0f_0)(\alpha(a_1)))(\r_0(s_0(a_3))f_0(a_2))\underline{B_9}\\
&&+\r_0((\partial_0f_0)(\alpha(a_1)))\r_0((\partial_0f_0)(a_3))f_0(a_2)\underline{B_{10}}\\
&&+\r_0((\partial_0f_0)(\alpha(a_1)))(f_0(a_2)\cdot \partial_0f_0(a_3))\underline{B_{11}}\\
&&+\r_0((\partial_0f_0)(\alpha(a_1)))f_0(a_2\cdot a_3)\underline{B_{12}},
\end{eqnarray*}
\begin{eqnarray*}
&&\r_0(\overline{s}_0(\alpha(a_3)))\big(\lambda_0(a_1,a_2)\big)\\
&=&\r_0((s_0-\partial_0f_0)(\alpha(a_3)))\big\{\r_0((s_0-\partial_0f_0)(a_1))f_0(a_2)\\
&&+\r_0((s_0-\partial_0f_0)(a_2))f_0(a_1)
-f_0(a_1)\cdot (\partial_0f_0)(a_2)-f_0(a_1\cdot a_2)\big\}\\
&=&\r_0(s_0(\alpha(a_3)))(\r_0(s_0(a_1))f_0(a_2))\underline{C_1}-\r_0(s_0(\alpha(a_3)))\r_0((\partial_0f_0)(a_1))f_0(a_2)\underline{C_2}\\
&&+\r_0(s_0(\alpha(a_3)))(\r_0(s_0(a_2))f_0(a_1))\underline{C_3}-\r_0(s_0(\alpha(a_3)))\r_0((\partial_0f_0)(a_2))f_0(a_1)\underline{C_4}\\
&&-\r_0(s_0(\alpha(a_3)))((f_0(a_1)\cdot \partial_0f_0(a_2))\underline{C_5}-\r_0(s_0(\alpha(a_3)))f_0(a_1\cdot a_2)\underline{C_6}\\
&&-\r_0((\partial_0f_0)(\alpha(a_3)))(\r_0(s_0(a_1))f_0(a_2))\underline{C_7}\\
&&+\r_0((\partial_0f_0)(\alpha(a_3)))\r_0((\partial_0f_0)(a_1))f_0(a_2)\underline{C_8}\\
&&-\r_0((\partial_0f_0)(\alpha(a_3)))(\r_0(s_0(a_2))f_0(a_1))\underline{C_9}\\
&&+\r_0((\partial_0f_0)(\alpha(a_3)))\r_0((\partial_0f_0)(a_2))f_0(a_1)\underline{C_{10}}\\
&&+\r_0((\partial_0f_0)(\alpha(a_3)))(f_0(a_1)\cdot \partial_0f_0(a_2))\underline{C_{11}}\\
&&+\r_0((\partial_0f_0)(\alpha(a_3)))f_0(a_1\cdot a_2)\underline{C_{12}},
\end{eqnarray*}
\begin{eqnarray*}
  &&\lambda_0(\alpha(a_1),a_2\cdot a_3)\\
&=& \r_0(\overline{s}_0(\alpha(a_1))f_0(a_2\cdot a_3)+\r_0(\overline{s}_0(a_2\cdot a_3))f_0(\alpha(a_1))\\
&&-f_0(\alpha(a_1))\cdot \partial_0f_0(a_2\cdot a_3))-\partial_0f_0(\alpha(a_1)\cdot (a_2\cdot a_3))\\
&=&\r_0(s_0(\alpha(a_1))f_0(a_2\cdot a_3)\underline{D_1}-\r_0((\partial_0f_0)(\alpha(a_1)))f_0(a_2\cdot a_3)\underline{D_2}\\
&&+\r_0(s_0(a_2\cdot a_3))f_0(\alpha(a_1))\underline{D_3}-\r_0((\partial_0f_0)(a_2\cdot a_3))f_0(\alpha(a_1))\underline{D_4}\\
&&-f_0(\alpha(a_1))\cdot \partial_0f_0(a_2\cdot a_3)\underline{D_5}-\partial_0f_0(\alpha(a_1)\cdot (a_2\cdot a_3))\underline{D_6},
\end{eqnarray*}
\begin{eqnarray*}
  &&\lambda_0(a_1\cdot a_2,\alpha(a_3))\\
&=& \r_0(\overline{s}_0(a_1\cdot a_2)f_0(\alpha(a_3))+\r_0(\overline{s}_0(\alpha(a_3)))f_0(a_1\cdot a_2)\\
&&-f_0(a_1\cdot a_2)\cdot \partial_0f_0(\alpha(a_3)))-\partial_0f_0(a_1\cdot a_2\cdot (\alpha(a_3)))\\
&=&\r_0(s_0(a_1\cdot a_2)f_0(\alpha(a_3))\underline{E_1}-\r_0((\partial_0f_0)(a_1\cdot a_2))f_0(\alpha(a_3))\underline{E_2}\\
&&+\r_0(s_0(\alpha(a_3)))f_0(a_1\cdot a_2)\underline{E_3}-\r_0((\partial_0f_0)(\alpha(a_3)))f_0(a_1\cdot a_2)\underline{E_4}\\
&&-f_0(a_1\cdot a_2)\cdot \partial_0f_0(\alpha(a_3)\underline{E_5}-\partial_0f_0(a_1\cdot a_2\cdot (\alpha(a_3)))\underline{E_6}.
\end{eqnarray*}
Due to \eqref{AssCommT0}, \eqref{a} and \eqref{cm1}, we have
$$\underline{A_1}-\underline{C_3}=0,~~\underline{A_2}+\underline{D_3}=0,~~\underline{A_3}+\underline{B_1}=0,
~~\underline{A_4}-\underline{E_1}=0,~~\underline{B_3}-\underline{C_1}=0,~~\underline{B_9}-\underline{C_2}=0,$$
$$\underline{B_2}-\underline{C_7}=0,~~\underline{D_6}-\underline{E_6}=0,~~\underline{B_7}-\underline{C_9}=0,
~~\underline{B_8}-\underline{C_8}=0,~~\underline{B_{10}}-\underline{C_{10}}=0,$$
$$-\underline{C_6}-\underline{E_3}=0,~~\underline{B_6}+\underline{D_1}=0,~~\underline{B_{12}}+\underline{D_2}=0,~~-\underline{C_{12}}-\underline{E_2}=0.$$
Thus we get
\begin{eqnarray*}
 &&(h_0-\overline{h}_0)(a_1,a_2,a_3)\\
 &=&-\r_0(s_0(\alpha(a_1)))\r_0((\partial_0f_0)(a_3))f_0(a_2)-\r_0(s_0(\alpha(a_1)))((f_0(a_2)\cdot \partial_0f_0(a_3))\\
&&+\r_0((\partial_0f_0)(\alpha(a_1)))(f_0(a_2)\cdot \partial_0f_0(a_3))+\r_0(s_0(\alpha(a_3)))\r_0((\partial_0f_0)(a_2))f_0(a_1)\\
&&+\r_0(s_0(\alpha(a_3)))((f_0(a_1)\cdot \partial_0f_0(a_2))-\r_0((\partial_0f_0)(\alpha(a_3)))(f_0(a_1)\cdot \partial_0f_0(a_2))\\
&&-\r_0((\partial_0f_0)(\alpha(a_1)))f_0(a_2\cdot a_3)-f_0(\alpha(a_1))\cdot \partial_0f_0(a_2\cdot a_3)\\
&&+\r_0((\partial_0f_0)(a_1\cdot a_2))f_0(\alpha(a_3))+f_0(a_1\cdot a_2)\cdot \partial_0f_0(\alpha(a_3))\\
&&+(d (\omega_0-\overline{\omega}_0-\lambda_0))(a_1,a_2,a_3).
\end{eqnarray*}
Using \eqref{cm1} again, we  have
 $(h_0-\overline{h}_0)(a_1,a_2,a_3)=d (\omega_0-\overline{\omega}_0-\lambda_0)(a_1,a_2,a_3)$.

\subsection*{Acknowledgements}
The authors would like to thank the referee for careful reading of the paper and for valuable comments.
This research was supported by NSFC(11501179,11601219).

\end{document}